\documentclass[11pt]{article}

\usepackage[english]{babel}
\usepackage{graphicx,url}
\usepackage{amsmath,amssymb,latexsym,graphics,epsfig}
\usepackage{hyperref}
\usepackage{color}
\usepackage{amsthm}

\newtheorem{theo}{Theorem}[section]

\newtheorem{lema}[theo]{Lemma}

\input amssym.def
\newsymbol\rtimes 226F
\newfont{\nset}{msbm10}

\def\C{\mbox{\boldmath $C$}}

\def\O{\mbox {\boldmath $O$}}

\def\1{\mbox{\rm 1}}

\def\d{\delta}

\def\vec0{\mbox{\boldmath $ 0 $}}
\def\w{\mbox{\boldmath $ w $}}
\def\x{\mbox{\boldmath $ x $}}

\def\z{\mbox{\boldmath $ z $}}

\def\1{\mbox{\boldmath $ 1 $}}

\def\j{\mbox{\boldmath $ j $}}

\def\A{\mbox{\boldmath $A$}}

\def\matrixM{\mbox{\boldmath $M$}}
\def\P{\mbox{\boldmath $ P $}}

\def\subC{\mbox{{\tiny $C$}}}
\def\subI{\mbox{{\tiny $I$}}}
\def\subF{\mbox{{\tiny $F$}}}

\begin{document}

\title{A Differential Approach for Bounding\\the Index of Graphs under Perturbations}

\author{C. Dalf\'{o}, M.A. Fiol, E. Garriga \\
         Departament de Matem\`{a}tica Aplicada IV \\
         Universitat Polit\`{e}cnica de Catalunya\\
         \texttt{\{cdalfo,fiol,egarriga\}@ma4.upc.edu}}

\maketitle

\begin{abstract}
This paper presents bounds for the variation of the spectral radius $\lambda(G)$ of a graph $G$ after
some perturbations or local vertex/edge modifications of $G$.
The perturbations considered here are the connection of a new vertex with, say, $g$ vertices of $G$, the addition of a pendant edge (the previous case with $g=1$) and the addition of an edge.
The method proposed here is based on continuous perturbations and the study of
their differential inequalities associated. Within rather economical information (namely, the degrees of the vertices involved in the perturbation), the best possible
inequalities are obtained. In addition, the cases when
equalities are attained are characterized. The asymptotic behavior of the bounds obtained is also discussed. For instance, if $G$ is a connected graph and $G_u$ denotes the graph obtained from $G$ by adding a pendant edge at vertex $u$ with degree $\delta_u$, then,
$$
\textstyle
\lambda(G_u)\le \lambda(G)+\frac{\delta_u}{\lambda^3(G)}+\textrm{o}\left(\frac{1}{\lambda^3(G)}\right).
$$
\end{abstract}

\noindent{\em Keywords:} Graph, Adjacency matrix, Spectral radius, Graph perturbation, Differential inequalities.
\\
\noindent{\em 2000 MSC:} 05C50 (47A55).

\section{Introduction}
When we represent a graph by its adjacency matrix, it is natural to ask how the
properties of the graph are related to the spectrum of the matrix. As is well-known, the spectrum does not characterize the graph, that is, there are nonisomorphic cospectral graphs. However, important properties of the graph stem from the knowledge of its spectrum. A summary of these relations
can be seen in Schwenk and Wilson~\cite{ScWi78} and, in a more extensive way, in Cvetkovi\'c,
Doob and Sachs~\cite{CvDoSa79} and Cvetkovi\'c, Doob, Gutman and Torga\v{s}ev~\cite{CvDoGuTo88}.

The perturbation of a graph $G$ is to be thought of as a local modification, such as the addition
or deletion of a vertex or an edge. The cases stu\-died here
are the addition of a vertex (together with some incident edges), an edge and a pendant edge. When we make the perturbation, the
spectrum changes and it is particulary interesting to study the behavior of the maximum eigenvalue $\lambda(G)$,
which is called spectral radius or index of $G$. For a comprehensive survey of results about this parameter, we refer the reader to Cvetkovi\'c  and Rowlinson~\cite{CvRo90}.
In particular, accurate bounds for $\lambda(G)$ were obtained, under some conditions, with the knowledge of the spectral radius, the associated eigenvector and the second eigenvalue. More details about these methods can be found in the survey by Rowlinson~\cite{Ro91}.

The bounds obtained here for the index of a perturbed graph are a bit less
precise than those in Rowlinson~\cite{Ro91}, but we believe that ours have two aspects of interest. First, they are derived from a mere knowledge of the degrees of the vertices involved in the perturbation. Second,
they are the best possible, in the sense that we characterize the cases in which the bounds are attained. Our approach is based on the study of some differential
inequalities, seeing the perturbation as a continuous process or, to be more precise, as a linear matrix perturbation.
Although the theory of matrix perturbations (see, for instance, the textbook by Stewart and Sun \cite{StSu90} or the chapter by Li \cite{Li}) has been commonly used in this context, to the authors' knowledge, our method has not been used before to bound the index of graphs under perturbations.

\section{Notation and Basic Concepts}
Our graphs are undirected, simple (without loops or multiple edges), connected and finite. The
graph $G=(V,E)$ has set of vertices $V$, with cardinality $n=|V|$, and set of edges $E$. The trivial graph  with only one vertex $u$ is denoted by $K_1=\{u\}$.
If  $G_1=(V_1,E_1)$ and $G_2=(V_2,E_2)$, then $G_1\cup G_2=(V_1\cup
V_2,E_1\cup E_2)$ and $G_1+G_2=(V_1\cup V_2,E_1
\cup E_2\cup E)$, where $E$ is the set of edges that join every vertex of $V_1$
with all vertices of $V_2$.
%A walk in  $G$ is a sequence  $v_1v_2\ldots v_k$ of
%vertices such that $v_iv_{i+1}$ is an edge
%for $i=1,\ldots,k-1$. The graph will be said to be connected if for any pair of
%vertices there exists a walk that connects them.
The adjacency matrix $\textbf{\emph{A}}=(a_{ij})$ of $G$ has entries $a_{ij}= 1$ if $u_iu_j\in E$ and $a_{ij}=0$ otherwise. We denote by $\textbf{\emph{j}}$ the (column) vector of $\mathbb{R}^n$ with all its
entries equal to $1$. Hence,  $\textbf{\emph{A}}\textbf{\emph{j}}$ is the vector of degrees $(\d_1,\d_2,\ldots,\d_n)^{\top}$. In
particular, $G$ is regular of degree $\d$ if and only if $\textbf{\emph{A}}\textbf{\emph{j}}=\d\textbf{\emph{j}}$.

A real matrix $\textbf{\emph{M}}=(m_{ij})$ is said to be nonnegative if $m_{ij}\geq 0$, for any $i,j$. We
say that $\matrixM$ is connected if, given any pair $i$ and $j$, there exists a sequence
$i_0,i_1,\ldots,i_r$ such that $i_0=i$, $i_r=j$ and $m_{i_{h-1}i_h}\neq
0$, for $h=1,2,\ldots,r$. Trivially, the adjacency matrix of a connected graph is symmetric, nonnegative
and connected.

The spectrum of a square matrix is the set of its eigenvalues in the complex plane.
The {\em spectral radius} is the maximum of the modulus of its eigenvalues.
If the matrix is the adjacency matrix of a graph, we call it the
{\em index} of the graph. A symmetric real matrix has only real eigenvalues,
which are numbered in nonincreasing order $\lambda_1\geq\lambda_2\geq\cdots\geq\lambda_{n}$.
Then, the spectral radius is the maximum of $|\lambda_1|$ and $|\lambda_{n}|$. Also,
the spectral radius can be defined as $\lambda=\textrm{sup}\left\{\|\textbf{\emph{A}}\textbf{\emph{x}}\|:\|\textbf{\emph{x}}\|=1\right\}$, and defines a norm
in the space of symmetric matrices. Then, $\|\textbf{\emph{A}}\textbf{\emph{u}}\|\leq \|\textbf{\emph{A}}\|\,\|\textbf{\emph{u}}\|$, for any vector $\textbf{\emph{u}}$, with
equality if and only if $\textbf{\emph{u}}$ is an eigenvector associated to an eigenvalue giving the spectral radius. For a connected nonnegative symmetric real matrix, the theorem of
Perron-Frobenius states the following:
\begin{enumerate}
  \item The first eigenvalue  equals the spectral radius $\lambda_1=\lambda$.
  \item The eigenvalue $\lambda_1$ is a simple root of the characteristic polynomial.
  \item There is a unitary eigenvector $\textbf{\emph{x}}$ corresponding to $\lambda_1$ with strictly positive entries.
\end{enumerate}

%We will then use the inequality \ds \sqrt\Delta \leq \lambda\leq \Delta\), where $\Delta$ is the
%maximum degree of graph [2].

%Farem œs desprŽs de la desigualtat:
%\ds \sqrt\Delta \leq \lambda\leq \Delta\), on $\Delta$ Žs el grau mˆxim del
%graf [2].

\section{General Technique}

Let ${\cal S}^+$ (respectively, ${\cal S}^+_{\subC}$) be the subset of symmetric, nonnegative
(respectively, and connected) matrices of the space $M(n,n)$ of real $n\times n$ matrices.

When a perturbation modifies a graph into another, we denote by $G_{\subI}$ the initial graph and
by $G_{\subF}$ the final graph. Similarly, if $\A_{\subI}$ and $\textbf{\emph{A}}_{\subF}$ are the adjacency matrices of the graphs $G_{\subI}$ and $G_{\subF}$ on $n$ vertices, we say that $\textbf{\emph{A}}_{\subF}$ is obtained from $\textbf{\emph{A}}_{\subI}$ by the perturbation
$\textbf{\emph{P}}=\textbf{\emph{A}}_{\subF}-\textbf{\emph{A}}_{\subI}$.

If $G_{\subF}$ is connected, then the matrices $\A(t)=\textbf{\emph{A}}_{\subI}+t\textbf{\emph{P}}$ belong to ${\cal S}^+_{\subC}$ for every
$t\in(0,1]$. Similarly, if $G_{\subI}$ is connected, then $\A(t)\in {\cal S}^+_{\subC}$ for
$t\in[0,1)$. Also, if $G_{\subF}$ is connected and the perturbation matrix $\textbf{\emph{P}}\in {\cal S}^+$,
then $\textbf{\emph{A}}_{\subI}+t\textbf{\emph{P}}\,\in {\cal S}^+_{\subC}$ for $t\in(0,\infty)$.

If $\textbf{\emph{A}}$ and $\textbf{\emph{P}}$ are symmetric matrices, there exist continuous real functions $\mu_1(t)$, $\mu_2(t)$, $\ldots$, $\mu_n(t)$, and continuous vectorial functions $\x_1(t),\x_2(t),\ldots,\x_n(t)$ that are,
respectively, the eigenvalues of $\A(t)=\textbf{\emph{A}}+t\textbf{\emph{P}}$ and their associated eigenvectors. From the implicit function
theorem, if $\mu_i(t_0)$ is a simple eigenvalue, then $\mu_i$ is a ${\cal C}^1$-function
in a neighborhood of $t_0$. Therefore, if $\A(t)\in {\cal S}^+_{\subC}$ for $t$ belonging to an
interval $I$, the spectral radius is a continuously differentiable function in $I$. In the
three results that we present, the perturbation matrix $\textbf{\emph{P}}$ belongs to ${\cal S}^+$ and the
perturbed matrix $\textbf{\emph{A}}_{\subF}=\textbf{\emph{A}}_{\subI}+\textbf{\emph{P}}$ to ${\cal S}^+_{\subC}$. Thus, the normalized positive eigenvector $\x(t)$ associated with the spectral radius $\lambda(t)$ of the matrix $\textbf{\emph{A}}(t)=\textbf{\emph{A}}_{\subI}+t\textbf{\emph{P}}$ is a ${\cal C}^1(0,\infty)$-function, which can be extended with continuity to $[0,\infty)$, but now $\x(t)$ might have lost the strictly positive character of its entries.

Our technique is based on the following result:

\begin{lema}
\label{lema0}
Let $\x(t)=(\alpha_1,\alpha_2,\ldots,\alpha_n)^{\top}$ be the normalized $\lambda(t)$-eigenvector of the matrix $\A(t)=\A_{\subI}+t\P$ with $\P=(p_{ij})$. Then,
\begin{equation}
  \label{fonamental}
  \lambda'=\langle\P\x,\x\rangle=\sum_{i,j=1}^{n}p_{ij}\alpha_i\alpha_j.
 \end{equation}
\end{lema}

\begin{proof}
By differentiating the expression $\textbf{\emph{A}}\x=(\A_{\subI}+t\P)\x=\lambda\x$, we get
$$
\P\x+\textbf{\emph{A}}\x'=\lambda'\x+\lambda \x'.
$$
Then, the result follows by taking the inner product by $\x$ and
observing that, from $\langle \x,\x\rangle=1$, we have
$\langle \x',\x\rangle=0$ and $\langle\A\x',\x\rangle= \langle \x',\A\x\rangle
=\lambda\langle \x',\x\rangle=0$.
\end{proof}

A first remark is that if  $\P\in {\cal S}^+$ and  $\A_{\subF}=\A_{\subI}+\P\in {\cal S}^+_{\subC}$, then the spectral radius increases strictly and, in particular, $\lambda_{\subI}=\lambda(0)<\lambda(1)=\lambda_{\subF}$. Also, since there exists
$\lim_{t\rightarrow 0^+}\lambda(t)=
\sum_{i,j=1}^{n}p_{ij}\alpha_i(0)\alpha_j(0)$, by the mean value theorem, we have that
$\lambda$ is also differentiable at $0$ with
$\displaystyle \lambda'(0)= \sum_{i,j=1}^{n}p_{ij}\alpha_i(0)\alpha_j(0)$.

We present three results of bounds of the index of a graph for the
following perturbations: connecting an isolated vertex, adding an edge and adding a pendant edge.
Starting from Eq.~(\ref{fonamental}), we give differential inequalities with
information on the degrees of the vertices involved, and we characterize the case when they become equations.
Solving these equations, we reach our conclusions by using the following result on differential
inequalities (see Szarski~\cite{Sz67}):

\begin{lema}
\label{lema1}
Let $A$ be an open convex subset of $\mathbb{R}^2$ and let $f:A\rightarrow\mathbb{R}$, $(t,x)\mapsto f(t,x)$, be a continuous function with $\frac{\partial f}{\partial x}$ continuous. Let $u,v:[t_0,\alpha)\rightarrow\mathbb{R}$ be continuously differentiable functions, such that:
\begin{enumerate}
\item
For all $t\in[t_0,\alpha)$, $(t,u(t))\in A$, $(t,v(t))\in A$.
\item
Function $u$ satisfies: $u'(t)=f(t,u(t))$ for all $t\in[t_0,\alpha)$, $u(t_0)=x_0$.
\item
Function $v$ satisfies: $v'(t)<f(t,v(t))$ for all $t\in(t_0,\alpha)$, $v(t_0)=x_0$, $v'(t_0)\leq f(t_0,v(t_0))$.
\end{enumerate}
Then, $v(t)<u(t)$ for all $\in(t_0,\alpha)$.
\end{lema}

%\vspace{.5 cm}

\section{Connection of an isolated vertex}
Our first result is on the change of the index of a graph when we connect an isolated vertex to some other vertices. For this case Rowlinson \cite{Ro92} computed the characteristic polynomial of the modified graph in terms of the characteristic polynomial of the initial graph and some entries of its idempotents (see also Cvetkovi\'c and Rowlinson \cite[p.90]{CvRo04} for a shorter proof).

\begin{theo}
Let $G_{\subI}=(V,E)$ be a graph with an isolated vertex $u$.
Given some vertices $v_1,v_2,\ldots,v_g$ different from $u$, we denote by
$G_{\subF}$ the graph $(V,E\cup\{uv_1,uv_2,\ldots,uv_g\})$, which is assumed to be
connected. If $\lambda_{\subI}$ and $\lambda_{\subF}$ are the spectral radii of
$G_{\subI}$ and $G_{\subF}$, respectively, then the following inequality holds:
$$
\lambda_{\subF}\leq H^{-1}(\lambda_{\subI}),
$$
where the function $H:(0,+\infty)\rightarrow \mathbb{R}$ is defined by $H(\xi)=\xi-\frac{g}{\xi}$.
The equality is satisfied if and only if $G_{\subF}=\{u\}+G$, with $G$ being a regular graph.
\end{theo}

\begin{proof}
Let $n+1$ be the order of the graphs $G_{\subI}$ and $G_{\subF}$.
The continuous perturbation of the matrix associated with $G_{\subI}$ that produces the matrix associated with $G_{\subF}$ can be described by
$$
\A(t)=\A_{\subI}+t\P=
            \left(\begin{array}{cccc}
                     0 & 0 & \cdots & 0 \\
                     0 & & & \\
                     \vdots & & \textbf{\emph{C}} & \\
                     0 & & &
                   \end{array}\right)
            + t
            \left(\begin{array}{cccc}
                     0 & \cdots & \w^{\top} & \cdots \\
                     \vdots & & & \\
                     \w & & \O & \\
                     \vdots & & &
                   \end{array}\right), \qquad t\in [0,1],
$$
where $\w$ is the column binary vector associated with the perturbation
and $\textbf{\emph{C}}$ is the adjacency matrix of graph $G_{\subI}-\{u\}$.
Note that, for any $t\in (0,1]$, the matrix $\A(t)$ is nonnegative and connected.
Let $\lambda(t)$ be the spectral radius of $\A(t)$. Let
$\x(t)=(\alpha|\z)^{\top}=(\alpha,z_1,z_2,\ldots,z_n)^{\top}$
be its normalized positive eigenvector. Then, by Eq. (\ref{fonamental}),
$$
\lambda'=\langle\P\x,\x\rangle=2\alpha\langle\z,\w\rangle.
$$
From $\A(t)\x(t)=\lambda (t)\x(t)$, we have
\begin{equation}
\label{Ax=lx}
\left(
\begin{array}{cc}
0 & t\w^{\top} \\
t\w & \C
\end{array}
\right)
\left(
\begin{array}{c}
\alpha \\
\z
\end{array}
\right)
=
\left(
\begin{array}{c}
t\langle \w, \z\rangle \\
t\alpha\w+\C\z
\end{array}
\right)
=
\left(
\begin{array}{c}
\lambda\alpha \\
\lambda\z
\end{array}
\right),
\end{equation}
and the first scalar equation  gives
\begin{equation}\label{desigualtatvertex1}
      \lambda^2\alpha^2=t^2\langle\z,\w\rangle^2\leq t^2\|\z\|^2g=
      t^2(1-\alpha^2)g.
\end{equation}
Hence,
\begin{equation}\label{desigualtatvertex2}
     \lambda'=2\lambda\frac{\alpha^2}{t}\leq\frac{2gt\lambda}{\lambda^2+gt^2}.
\end{equation}
The inequalities $(\ref{desigualtatvertex1})$ and $(\ref{desigualtatvertex2})$ are either equalities or strict inequalities in the whole interval $(0,1]$.
%are either equalities or strict inequalities.
Indeed, if the equalities are satisfied for $t_0$,
then $\z(t_0)$, which has only positive entries, would be
proportional to $\w$, which is not null. Therefore, $\w=\j$ and $\z(t_0)=\beta\textbf{\emph{j}}$.
Hence, at $t=t_0$ the last $n$ equations of
(\ref{Ax=lx})  become
$\textbf{\emph{C}}\j=\left(\lambda-t_0\frac{\alpha}{\beta}\right)\j$, where $\alpha=\alpha(t_0)$, so that $G_{\subI}=\{u\}\cup G$, $G_{\subF}=\{u\}+G$, and $G$ is a regular graph. To conclude that, in this situation,
$(\ref{desigualtatvertex2})$ is an equality for all $t\in(0,1]$, let us study the existence of solutions to the following system:
$$
\left(\begin{array}{cccc}
             0 & t & \cdots & t \\
              t & & & \\
              \vdots & & \textbf{\emph{C}} &\\
              t & & &
            \end{array}\right)
      \left(\begin{array}{c}
              \alpha \\
              \beta \\
              \vdots \\
              \beta
            \end{array}\right)=\lambda
      \left(\begin{array}{c}
              \alpha \\
              \beta \\
              \vdots \\
              \beta
            \end{array}\right), \qquad
      \alpha^2+n\beta^2=1.
$$
Then, for all $t$, we obtain the solution:
$$
\lambda=\frac{\delta}{2}+\sqrt{\frac{\delta^2}{4}+nt^2},\qquad
        \alpha=\sqrt{\frac{\lambda-\delta}{2\lambda-\delta}},\qquad
        \beta=\sqrt{\frac{\lambda}{n(2\lambda-\delta)}},
$$
where $\delta=\lambda-t\frac{\alpha}{\beta}$ denotes the degree of $G$.

%\begin{figure}[htb]
%        \vspace{8.5 cm}\hspace{1.8 cm}
%        \special{postscriptfile DIB1art1.ps}
%        \caption{Bound when an isolated vertex is joined.}
%      \end{figure}

%\vspace{.5 cm}

Now we have the following cases, where $f(t,\lambda)=\frac{2gt\lambda}{\lambda^2+gt^2}$:
\begin{itemize}
 \item[$(a)$]
 $\lambda'=f(t,\lambda)$ for all $t\in [0,1]$,
 $\lambda(0)=\lambda_{\subI}$, if $G_{\subF}=\{u\}+ G$, with
 $G$ being a regular graph.
 \item[$(b)$]
 $\lambda'<f(t,\lambda)$ for all $t\in (0,1]$,
 $\lambda'(0)=f(0,\lambda(0))$,
 $\lambda(0)=\lambda_{\subI}$, in any other case.
\end{itemize}

The Cauchy problem
$$
y'=\frac{2gty}{y^2+gt^2}, \qquad y(0)=\lambda_{\subI},
$$
can be solved by making the changes $y=\sqrt{rs}$ and $t=\sqrt{s}$, so giving
$$
y^2(t)-\lambda_{\subI}y(t)-gt^2=0.
$$
Hence,
$$
y(1)-\frac{g}{y(1)}=\lambda_{\subF}-\frac{g}{\lambda_{\subF}}=\lambda_{\subI}
$$
and, introducing the bijection $H:(0,+\infty)\rightarrow\mathbb{R}$,
$\displaystyle H(\xi)=\xi-\frac{g}{\xi}$, the theorem follows from Lemma~\ref{lema1}.
\end{proof}

\section{Addition of an edge}
The second result that we present is on the change of the index when we add an edge to a graph.
In this context, Rowlinson \cite{Ro88} proved that, under some conditions, the index of the perturbed graph can be determined by the eigenvalues of the original graph together with some of its angles.
Moreover, some upper and lower bounds for such an index were given by Maas \cite{Ma87}.

\begin{theo}
Let $G_{\subI}=(V,E)$ be a graph and let $u,v \in V$ be two nonadjacent vertices with degrees $\delta_u,\delta_v$.
Let $G_{\subF}=(V,E \cup \{uv\})$, which we assume to be connected.
If $\lambda_{\subI}$ and $\lambda_{\subF}$ are, respectively, the indices of
$G_{\subI}$ and $G_{\subF}$, then
$$
\lambda_{\subF}\leq1+K^{-1}(K(\lambda_{\subI})-1),
$$
where $K:(0,\infty)\rightarrow\mathbb{R}$ is defined as
$K(\xi)=\xi-\frac{\delta_u+\delta_v}{\xi}$.
The equality is satisfied if and only if $G_{\subI}=(\{u\}\cup\{v\})+G$, where $G$
is a regular graph.
\end{theo}

\begin{proof}
       Let $n+2$ be the order of graphs $G_{\subI}$ and $G_{\subF}$ with adjacency matrices $\A_{\subI}$ and $\A_{\subF}$,
    respectively.
    In the language of perturbations, we can consider that $A_{\subI}$ and $A_{\subF}$ are related by $\A_{\subF}=\A_{\subI}+\P$, where
    $\P=(p_{ij})$ has entries $p_{12}=p_{21}=1$ and $p_{ij}=0$ otherwise (if necessary, we rearrange the
    vertices so that
    $v_1=u$ and $v_2=v$). Considering the continuous perturbation, let us consider the uniparametric family of
    matrices
    $$
    \A(t)=\A_{\subI}+t\P=
      \left(
      \begin{array}{ccccc}
       0 & t & \cdots &  \w_{u}^\top  & \cdots \\
       t & 0 & \cdots &  \w_{v}^\top  & \cdots \\
       \vdots & \vdots & & & \\
        \w_u &  \w_v & & \textbf{\emph{C}} & \\
       \vdots & \vdots & & &
      \end{array}
      \right),
      \qquad t\in [0,1],
    $$
where $\w_{u},\w_{v}\in \{0,1\}^n$ and $\textbf{\emph{C}}$ is the $n\times n$ adjacency matrix of the subgraph $G_{\subI}-\{u\}-\{v\}$.

Let $\lambda(t)$ be the spectral radius of $\A(t)$, which is a continuous function on $t$
for $t\in [0,1]$, and is differentiable for $t\in (0,1]$ by the connectedness of $\A(t)$.

Now, with $\x(t)=(\alpha,\beta|\z)^{\top}=(\alpha,\beta,z_1,z_2,\ldots,z_n)^{\top}$, Eq. $(\ref{fonamental})$ becomes
$$
\lambda'=\langle\P\x,\x\rangle=2\alpha\beta.
$$

Considering the first two entries of $(\lambda(t)\textbf{\emph{I}}-\A(t))\x(t)=\vec0$,
we get the system
$$
   \textbf{\emph{M}} \left(\begin{array}{c}
                  \alpha \\
                  \beta
                \end{array}
                \right) =
                \left(
                \begin{array}{c}
                r \\
                s
                \end{array}
                \right),
 $$
 with
 $$
 \textbf{\emph{M}}=\left(\begin{array}{cc}
 \lambda & -t \\
 -t      & \lambda
 \end{array}
 \right),
 \qquad
 r=\langle \w_{u},\z\rangle, \qquad s=\langle \w_{v},\z\rangle.
 $$
 Introducing the angles $\varphi_u$ and $\varphi_v$ that the vectors
 $\w_u$ and $\w_v$ form with $\z$, we can write
\begin{eqnarray}
\alpha^2+\beta^2 &=& \left\|\textbf{\emph{M}}^{-1}\left(
\begin{array}{c}
r \\
s
\end{array}\right)\right\|^2\leq
\left\|\textbf{\emph{M}}^{-1}\right\|^2(r^2+s^2) \nonumber \\
&=& \|\z\|^2 \frac{(\delta_u\cos^2\varphi_u+\delta_v\cos^2\varphi_v)}
{(\lambda-t)^2}\nonumber \\
& \leq & \frac{1-\alpha^2-\beta^2}{(\lambda-t)^2}(\delta_u+\delta_v),
\label{desigualtat21}
\end{eqnarray}
since $\frac{1}{\lambda-t}$ is the maximum eigenvalue of $\textbf{\emph{M}}^{-1}$ (to which eigenvector $(1,1)$ is associated).
Then,
\begin{equation}
2\alpha\beta
\leq
\alpha^2+\beta^2
\leq
\frac{\delta_u+\delta_v}{(\lambda -t)^2+\delta_u+\delta_v}.
\label{desigualtat22}
\end{equation}

Therefore, the spectral radius of $\A(t)$ satisfy the following differential inequality:
\begin{equation}
\lambda'\leq\frac{\delta_u+\delta_v}{(\lambda-t)^2+\delta_u+\delta_v}, \qquad
\lambda(0)=\lambda_{\subI}.\label{inequaciodiferencial2}
\end{equation}

We now prove that, in the interval $(0,1]$, expression $(\ref{inequaciodiferencial2})$ is
always an equality or a strict inequality. Let us assume that there exists
$t_0\in(0,1]$ such that $(\ref{inequaciodiferencial2})$ is an equality. Observing $(\ref{desigualtat22})$,
we see that the first inequality
is equivalent to $\alpha=\beta$ and the second one
to both equalities in
$(\ref{desigualtat21})$.
The first one occurs if $\sqrt{\delta_u}\cos\varphi_u=\sqrt{\delta_v}\cos\varphi_v$ and the
second if $\cos\varphi_u=\cos\varphi_v=1$. Therefore, the equality in
$(\ref{inequaciodiferencial2})$ is valid for a value $t_0$ when the
following conditions are simultaneously satisfied:
$$
\delta_u=\delta_v, \qquad \cos\varphi_u=\cos\varphi_v=1, \qquad \alpha=\beta.
$$

As all the entries of $\z$ are different from zero and $\w_u,\w_v$
are not null vectors, then it follows that $\w_{u}= \w_{v}=\textbf{\emph{j}}$ and $\x(t_0)=
(\alpha,\alpha,\gamma,\stackrel{(n)}{\ldots},\gamma)^{\top}$. The last
$n$ entries of $\A(t_0)\x(t_0)=\lambda\x(t_0)$ give
$2\alpha\textbf{\emph{j}}+\gamma\textbf{\emph{C}}\textbf{\emph{j}}=\lambda\gamma\textbf{\emph{j}}$,
that is,
$$
\textbf{\emph{C}}\textbf{\emph{j}}=\left(\lambda-2\frac{\alpha}{\gamma}\right)\textbf{\emph{j}},
$$
which means that $G_{\subI}=(\{u\}\cup\{v\})+G$, with $G$ being a regular graph with adjacency
matrix $\textbf{\emph{C}}$. Then, for all $t\in(0,1]$, there exist positive integers
$\alpha,\gamma,$ such that $(\alpha,\alpha,\gamma,
\stackrel{(n)}{\ldots},\gamma)^{\top}$ is an eigenvector (since all its entries are positive, it
corresponds to the spectral radius). Indeed, the system
$$
\left(
\begin{array}{ccccc}
0 & t & \cdots & \textbf{\emph{j}}^\top & \cdots \\
t & 0 & \cdots & \textbf{\emph{j}}^\top  & \cdots \\
\vdots & \vdots & & & \\
\textbf{\emph{j}} & \textbf{\emph{j}} & & \textbf{\emph{C}} & \\
\vdots & \vdots & & &
\end{array}
\right)
\left(
\begin{array}{c}
\alpha\\
\alpha\\
\gamma\\
\vdots\\
\gamma
\end{array}
\right)
=
\lambda
\left(
\begin{array}{c}
\alpha\\
\alpha\\
\gamma\\
\vdots\\
\gamma
\end{array}
\right), \qquad
2\alpha^2+n\gamma^2=1,
$$
has solution
\begin{eqnarray*}
\alpha & = & \frac{1}{2}\sqrt{1-\frac{\delta-t}{\sqrt{(\delta-t)^2+8n}}}, \\
\gamma & = & \frac{1}{\sqrt{2n}}\sqrt{1+\frac{\delta-t}{\sqrt{(\delta-t)^2+8n}}}\\
\lambda & = & \frac{1}{2}\left(t+\delta+\sqrt{(\delta-t)^2+8n}\right),
\end{eqnarray*}
where $\delta$ is the degree of $G$, and inequalities
$(\ref{desigualtat21})$ and $(\ref{desigualtat22})$ are equalities for all $t\in (0,1]$.
Extending by continuity to $[0,1]$, we have the following possibilities:
\begin{itemize}
\item[$(a)$]
$\lambda'=f(t,\lambda)$, for all $t\in[0,1]$, $\lambda(0)=\lambda_{\subI}$ if $G_{\subI}=(\{u\}\cup\{v\})+G$, with $G$ being regular;
\item[$(b)$]
$\lambda'<f(t,\lambda)$, for all $t\in(0,1]$, $\lambda'(0)\leq f(0,\lambda(0))$,
$\lambda(0)=\lambda_{\subI}$, in any other case;
\end{itemize}
where $f$ is the right side of differential inequality $(\ref{inequaciodiferencial2})$.

      %\begin{figure}[htb]
%        \vspace{10cm}\hspace{1.5 cm}
%        \special{postscriptfile dib2art1.ps}
%        \caption{bound when an edge is added}
%      \end{figure}
%\vspace{.5 cm}

Now, the solution to Cauchy's problem
$$
y'=\frac{\delta_u+\delta_v}{(y-t)^2+\delta_u+\delta_v}, \qquad y(0)=\lambda_{\subI},
$$
is
$$
y-\frac{\delta_u+\delta_v}{y-t}=\lambda_{\subI}-\frac{\delta_u+\delta_v}{\lambda_{\subI}}.
$$
By introducing the invertible function
$$
K:(0,\infty)\rightarrow \mathbb{R}, \qquad
K(\xi)=\xi-\frac{\delta_u+\delta_v}{\xi},
$$
we can write $y(1)=1+K^{-1}(K(\lambda_{\subI})-1)$.

Lemma $\ref{lema1}$ applied to case $(b)$ completes the proof.
\end{proof}

\section{Addition of a pendant edge}
The last result presented here is on the change of the index of a graph $G$ when we add a pendant edge to one of its vertices.
In this context, Bell and Rowlinson~\cite{BeRo88} derived, under certain conditions, exact values for the index of the perturbed graph in terms of the spectrum and certain angles of $G$.

\begin{theo} \label{teoremapendantedge}
Let $G_{\subI}=(V,E)$ be a connected graph, let
$u\in V$ be a vertex of degree $\delta_u$ and take a vertex $v\not\in V$. Let
$G_{\subF}=(V\cup\{v\},E\cup\{uv\})$. If $\lambda_{\subI}$ and $\lambda_{\subF}$ are the spectral radii of
$G_{\subI}$ and $G_{\subF}$ respectively, then
$$
\lambda_{\subF}\leq L_2^{-1}L_1(\lambda_{\subI}),
$$
where $L_1:(0,+\infty)\rightarrow\mathbb{R}$ is
$L_1(\xi)=\xi-\frac{g\delta_u}{\xi}$
and
$L_2:(1,+\infty)\rightarrow\mathbb{R}$
is
$L_2(\xi)=\xi-\frac{g\delta_u}{\xi-\frac{1}{\xi}}$.
The equality is satisfied if and only if $G_{\subI}=\{u\}+G$, with $G$ being a regular graph.
\end{theo}

\begin{proof}
      Let $n+1$ be the order of $G_{\subI}$.
   Rearranging the vertices suitably, the perturbation matrix
   $\P=(p_{ij})$ has $p_{12}=p_{21}=1$
   and the other entries are zero. Let us consider the matrices
   $$
   \A(t) =
      \left(
      \begin{array}{ccccc}
        0 & t & 0 & \cdots & 0 \\
        t & 0 & \cdots & \w^\top & \cdots\\
        0 & \vdots & & & \\
        \vdots & \w & & \textbf{\emph{C}} &  \\
        0 & \vdots &   &   &
      \end{array}
      \right), \qquad
      t\in[0,1],
      $$
such that $\A(0)$ is the adjacency matrix of the graph $G_{\subI}\cup\{u\}$, with the same spectral radius as $G_{\subI}$.

Now Eq. $(\ref{fonamental})$ becomes
   $$
   \lambda'=\langle\P\x,\x\rangle=2\alpha\beta,
   $$
where $\x(t)=(\alpha,\beta|\z)^{\top}$, with $\z^{\top}=(z_1,z_2,\ldots,z_n)^{\top}$ being
the normalized positive eigenvector, $t\in(0,1)$. The first two entries of the matrix equation $(\lambda(t)\textbf{\emph{I}}-\A(t))\x(t)=\vec0$ give the system
\begin{eqnarray*}
\lambda\alpha-t\beta &=& 0,\\
-t\alpha+\lambda\beta &=& \langle \w,\z\rangle.
\end{eqnarray*}
Introducing the angle $\varphi$ determined by $\z$ and $\w$, we can express the solution by
\begin{eqnarray*}
\alpha &=& \sqrt{\delta_u}\frac{\|\z\|}{\lambda^2-t^2}t\cos\varphi,\\
\beta &=& \sqrt{\delta_u}\frac{\|\z\|}{\lambda^2-t^2}\lambda\cos\varphi.
\end{eqnarray*}
Hence, using $\alpha^2+\beta^2+\|\z\|^2=1$, we obtain
$$
\lambda'=\frac{2\delta_ut\lambda\cos^2\varphi}{\delta_u(\lambda^2+t^2)\cos^2\varphi+(\lambda^2-t^2)^2}.
$$
   The constraint $\cos^2\varphi\leq 1$ implies that
   \begin{equation}\label{desigualtat4}
     \lambda'\leq\frac{2\lambda t\delta_u}{(\lambda^2-t^2)^2+\delta_u(\lambda^2+t^2)}
   \end{equation}
   for all $t\in(0,1]$.
   Let us observe that the continuous extension of $(\ref{desigualtat4})$
   to $t=0$ gives an equality, since $\alpha(0)=0$.

We now prove that inequality $(\ref{desigualtat4})$ is either an equality or a
strict inequality in the interval $(0,1]$. Indeed, if there existed $t_{0}\in(0,1]$ for which
$(\ref{desigualtat4})$ were an equality, then $\z(t_{0})$ and $\w$ would be proportional. As
all the entries of $\z$
   are strictly positive and $\w$ is not a null vector, then $\w=\textbf{\emph{j}}$ and $\z(t_{0})=\delta\textbf{\emph{j}}$. The last
$n$ equations of $(\lambda(t_0)\textbf{\emph{I}}-\A(t_0))\x(t_0)=\vec0$ give
   $ \textbf{\emph{C}}\textbf{\emph{j}}=\left(\lambda-\frac{\beta}{\delta}\right)\textbf{\emph{j}}$.
   Therefore, the graph
$G_{\subI}$ is $\{u\}+G$, with $G$ being a regular graph of degree $\delta=\lambda-\frac{\beta}{\delta}$ and with adjacency matrix $\textbf{\emph{C}}$.
   Then, $\z=\delta\textbf{\emph{j}}$ and, therefore, it is proportional to $\w=\textbf{\emph{j}}$, for all $t\in(0,1]$.
Indeed, the system
   $$
   \left(
      \begin{array}{ccccc}
        0 & t & 0 & \cdots & 0 \\
        t & 0 & \cdots & \textbf{\emph{j}}^\top  & \cdots\\
        0 & \vdots &   &   &       \\
        \vdots & \textbf{\emph{j}} & & \textbf{\emph{C}} &  \\
        0 & \vdots &   &   &
      \end{array}
      \right)
      \left(
      \begin{array}{c}
        \alpha\\
        \beta\\
        \gamma\\
        \vdots\\
        \gamma
      \end{array}
      \right)
      = \lambda
      \left(
      \begin{array}{c}
        \alpha\\
        \beta\\
        \gamma\\
        \vdots\\
        \gamma
      \end{array}
      \right), \qquad
      \alpha^2+\beta^2+n\gamma^2=1,
      $$
      gives the eigenvector of strictly positive entries
      $$
      \alpha=\frac{(\lambda-\delta)t}{\sqrt{\Lambda}}, \qquad
      \beta=\frac{(\lambda-\delta)\lambda}{\sqrt{\Lambda}}, \qquad
      \gamma=\frac{\lambda}{\sqrt{\Lambda}},
      $$
      where $\Lambda=2(n+t^2)\lambda^2-\delta(n+t+3t^2)\lambda+2t^2\delta^2$ and
      $\lambda$ is the maximum root of the polynomial
      $\lambda^3-\delta\lambda^2-(n+t^2)\lambda+\delta t^2$.
      By continuity, we thus have the two following possibilities:
\begin{itemize}
\item[$(a)$]
$\lambda'=f(t,\lambda)$, for all $t\in[0,1]$, $\lambda(0)=\lambda_{\subI}$
if $G_{\subI}=\{u\}+G$, with $G$ being a regular graph,
\item[$(b)$]
$\lambda'<f(t,\lambda)$, for all $t\in(0,1]$, $\lambda'(0)\leq f(0,\lambda(0))$, $\lambda(0)=\lambda_{\subI}$, in any other case,
\end{itemize}
where $f$ is the right side of differential inequality
$(\ref{desigualtat4})$.

%\begin{figure}[htb]
%        \vspace{7.5 cm}\hspace{1 cm}
%        \special{postscriptfile dib3art1.ps}
%        \caption{bound when a pendant edge is added}
%      \end{figure}
%
%      \vspace{.5 cm}

The differential equation
$$
y'= 2\delta_u\frac{ty}{(y^2-t^2)^2+\delta_u(y^2+t^2)},
$$
with initial condition $y(0)=\lambda_{\subI}\,$, is transformed into a linear equation
by means of the change:
$y=\sqrt{\frac{S+R}{2}}, t=\sqrt{\frac{S-R}{2}}$.
Solving it, we calculate implicitly $y(1)$,
represented by $\nu$, as one root of the equation
$$
(\nu^2+1)(\nu^2-1-\delta_u)^2+(\nu^2-1)^3+2\left(\delta_u -\lambda_{\subI}^2-\frac{\delta_u^2}
{\lambda_{\subI}^2}\right) (\nu^2-1)^2+\delta_u^2(\nu^2-1)=0,
$$
which may be factorized into the following two cubic equations:
\begin{eqnarray*}
& & \nu^3-\left(\lambda_{\subI}-\frac{\delta_u}{\lambda_{\subI}}\right)\nu^2-
        (\delta_u+1)\nu+\left(\lambda_{\subI}-\frac{\delta_u}{\lambda_{\subI}}\right)=0,\\
& & \nu^3+\left(\lambda_{\subI}-\frac{\delta_u}{\lambda_{\subI}}\right)\nu^2-
(\delta_u+1)\nu-\left(\lambda_{\subI}-\frac{\delta_u}{\lambda_{\subI}}\right)=0.
\end{eqnarray*}
The three roots of both equations are real, but only one in the first equation satisfies the necessary
condition $\nu\geq\sqrt{\delta_u+1}$. Introducing the bijective functions
$$
L_1:(0,+\infty) \rightarrow \mathbb{R}, \quad
L_1(\xi)=\xi-\frac{\delta_u}{\xi}, \qquad
L_2:(1,+\infty) \rightarrow \mathbb{R}, \quad
L_2(\xi)=\xi-\frac{\delta_u}{\xi-\frac{1}{\xi}},
$$
we can express $y(1)=L_2^{-1} L_1(\lambda_{\subI})$.
As before, Lemma $\ref{lema1}$ applied to case $(b)$ completes the proof.
      \end{proof}

  %\end{provateo}
  %\vspace{.5 cm}

\section{Asymptotic behavior}

It is illustrative to compare the bounds obtained in the three above theorems for graphs with
large index. Making the corresponding asymptotic developments, we have the following cases:
\begin{itemize}
\item[$(a)$]
Connection of an isolated vertex (to $g$ vertices):
$$
\lambda_{\subF}\leq H^{-1}(\lambda_{\subI})= \lambda_{\subI}+\frac{g}{\lambda_{\subI}}+\textrm{o}\left(
\frac{1}{\lambda_{\subI}}\right).
$$
\item[$(b)$]
Addition of an edge (between vertices of degrees $\delta_u,\delta_v$):
$$
\lambda_{\subF}\leq 1+K^{-1}(K(\lambda_{\subI})-1)=
\lambda_{\subI}+\frac{\delta_u+\delta_v}{\lambda_{\subI}^2}+\textrm{o}\left(
\frac{1}{\lambda_{\subI}^2}\right).
$$
\item[$(c)$]
Addition of a pendant edge (to a vertex of degree $\delta_u$):
$$
\lambda_{\subF}\leq L_2^{-1}L_1(\lambda_{\subI})= \lambda_{\subI}+\frac{\delta_u}{\lambda_{\subI}^3}+\textrm{o}\left(
\frac{1}{\lambda_{\subI}^3}\right).
$$
  \end{itemize}

Let us observe that the maximum possible variation in the spectral radius caused by the three
perturbations considered are, for large $\lambda_{\subI}$, of different orders of magnitude.

Notice also that, by applying iteratively the above formulas, we can obtain asymptotic bounds for `multiple perturbations'. For instance, if $G_{\subF}$ is obtained from $G_{\subI}$ by joining all the vertices $u_1,u_2,\ldots,u_m$ of a coclique, with respective degrees $\delta_1,\delta_2,\ldots,\delta_m$, we get, by applying the bound for the addition of an edge ${m\choose 2}$ times,
$$
\lambda_{\subF}\leq
\lambda_{\subI}+\frac{m-1}{\lambda_{\subI}^2}\sum_{i=1}^m \delta_i+\textrm{o}\left(
\frac{1}{\lambda_{\subI}^2}\right).
$$

\vskip 1cm
\noindent{\large \bf Acknowledgments.} The authors are most grateful to Professor Peter Rowlinson for his useful comments and suggestions on the topic of this paper. Research supported by the {\em Ministerio de Ciencia e Innovaci\'on}, Spain, and the {\em European Regional Development Fund} under project MTM2008-06620-C03-01 and by the {\em Catalan Research Council} under project 2009SGR1387.

\end{document}